\documentclass[11pt]{article}
\usepackage[utf8]{inputenc}
\usepackage[T1]{fontenc}
\usepackage[a4paper,margin=2.9cm]{geometry}
\usepackage{amsmath,amssymb,amsthm}
\usepackage{enumitem}
\usepackage{float}
\usepackage[expansion=false]{microtype}
\usepackage[colorlinks=true,linkcolor=black,citecolor=black,urlcolor=blue]{hyperref}

\allowdisplaybreaks
\setlist{nosep,leftmargin=*}

\theoremstyle{plain}
\newtheorem{theorem}{Theorem}[section]
\newtheorem{proposition}[theorem]{Proposition}
\newtheorem{lemma}[theorem]{Lemma}
\newtheorem{corollary}[theorem]{Corollary}
\newtheorem{conjecture}[theorem]{Conjecture}

\theoremstyle{remark}
\newtheorem{remark}[theorem]{Remark}

\newcommand{\Z}{\mathbb{Z}}
\newcommand{\Q}{\mathbb{Q}}
\newcommand{\Hur}{\mathcal{H}}
\DeclareMathOperator{\ord}{ord}
\title{\bfseries Modular periodicity of the Euler up/down numbers\\
at odd prime powers}

\author{Berke G\"ule\c{c}\\[2pt]
\small Departments of Control and Automation Engineering and Electronics and Communication Engineering,\\
\small Faculty of Electrical and Electronics Engineering, Istanbul Technical University,\\
\small 34469 Maslak, Istanbul, T\"urkiye\\
\small \texttt{gulecb21@itu.edu.tr}}
\date{}
\begin{document}
\maketitle

\begin{abstract}
\noindent
Let $E_n$ denote the number of alternating permutations of $\{1,\dots,n\}$,  equivalently characterized by
 $\sum_{n\ge0}E_nz^n/n!=\sec z+\tan z$. For every $q\ge1$, the sequence
$(E_n\bmod q)_{n\ge0}$ is eventually periodic; let $d(q)$ and $s(q)$ denote its minimal
eventual period and preperiod. For every odd prime $p$, Knuth and Buckholtz proved
$d(p)=\operatorname{lcm}(p-1,4)$ together with
\[
d(p^r)\mid p^{r-1}d(p),
\qquad
s(p^r)\le r,
\]
and Ramassamy conjectured that both bounds are attained for every $r\ge1$.

In this paper, we introduce an algebraic frequency expansion for the Euler zig-zag
numbers using Hurwitz series over the coefficient ring
$S_r=(\mathbb Z/p^r\mathbb Z)[x]/(x^2+1)$. More precisely, the corresponding
Hurwitz series is represented as a finite combination of formal exponential modes,
in a manner reminiscent of Fourier analysis.Using this expansion, we prove 
\[
d(p^r)=p^{r-1}d(p)
\qquad
\text{for every odd prime $p$ and every $r\ge1$},
\]
thereby establishing Ramassamy's period conjecture. We also disprove the preperiod
conjecture by proving
\[
s(5^5)=4<5.
\]
Finally, we prove that $5^5$ is the smallest odd prime power for which
$s(p^r)\ne r$, and based on our findings we further conjecture 
\[
s(p^r)\ge r-2
\]
for every odd prime $p$ and every $r\ge2$.
\end{abstract}


\section{Introduction}
For $n\ge0$, let $E_n$ denote the number of alternating permutations of
$\{1,\dots,n\}$. Andr\'e~\cite{Andre} showed that
\begin{equation}\label{eq:andre}
\sum_{n\ge0}E_n\frac{z^n}{n!}=\sec z+\tan z.
\end{equation}
The Euler up/down (also known as zig-zag) numbers $(E_n)_{n\ge0}$ have subsequently been studied from
various combinatorial and arithmetic viewpoints; see, for example,
Entringer~\cite{Entringer}, Millar, Sloane and Young~\cite{MSY},
Arnold~\cite{Arnold}, Brown, Fink and Willbrand~\cite{BrownFinkWillbrand},
and Stanley~\cite{Stanley}.

Knuth and Buckholtz~\cite{KB} proved that, for every integer $q\ge1$, the
sequence $(E_n\bmod q)_{n\ge0}$ is eventually periodic. We write $d(q)$ for
its minimal eventual period and $s(q)$ for its preperiod. For every odd prime
$p$, they further proved
\[
d(p)=\operatorname{lcm}(p-1,4),
\]
and
\begin{equation}\label{eq:kb}
d(p^r)\mid p^{r-1}d(p),
\qquad
s(p^r)\le r.
\end{equation}
Jha~\cite{Jha} later gave an alternative argument from which the bounds in
\eqref{eq:kb} follow.Ramassamy~\cite{R} conjectured that both bounds are attained:
\begin{equation}\label{eq:conj}
d(p^r)=p^{r-1}d(p),
\qquad
s(p^r)=r
\qquad(r\ge1).
\end{equation}
Our main result proves the period identity in \eqref{eq:conj} for every odd
prime $p$ and every $r\ge1$, while the preperiod identity is disproved by the
example $s(5^5)=4<5$, which turns out to be the smallest counterexample.

Since the problem concerns arithmetic modulo $p^r$, the ring $\Z/p^r\Z$ is a
natural starting point. At the same time, the trigonometric generating function
$\sec z+\tan z$ naturally gives us an intuition to work with imaginary numbers so that these trigonometric terms can be represented in another way. Thus, in
a way analogous to the passage from $\mathbb R$ to $\mathbb C$, we work in
\[
S_r:=(\Z/p^r\Z)[x]/(x^2+1),
\]
with $i=[x]$ and $i^2=-1$. In particular, $i$ has order $4$, naturally
reflecting the factor $4$ appearing in
$d(p)=\operatorname{lcm}(p-1,4)$, which we make extensive use of.

Then, we developed a tool called the \emph{frequency expansion}, whose main
idea is inspired by finite frequency decompositions familiar from digital
signal processing. To do so, we encode the Euler sequence as a Hurwitz series
$F\in S_r\langle z\rangle$. A finite frequency expansion of $F$ then yields,
at the coefficient level,
\[
F_n=\sum_k H_k(ki)^n.
\]
where the coefficients $H_k$ play a role similar to Fourier coefficients. This expansion will play a critical role and forms the basis of our findings.
We now briefly explain the role of each section.

Section~\ref{sec:periodic} develops the basic facts about eventual periods and
preperiods, Section~\ref{sec:ring} introduces the algebraic setting we mentioned above, and
Section~\ref{sec:hurwitz} constructs the frequency expansion $F=(F_n)_{n\geq0}$ formally. In
Section~\ref{sec:period}, this expansion first gives an alternative proof of the
Knuth--Buckholtz formula independently and the expansion is then used to prove the period identity
in \eqref{eq:conj}.The key step that we will see at higher prime powers is to pass from $S_r$ to
$S_r/pS_r$, where frequencies that are indistinguishable modulo $p$ can be
grouped together, the details of which are explained in the relevant part.

Section~\ref{sec:preperiod} develops
a divisibility criterion that determines the exact preperiod through
successive tests involving the Euler numbers. In particular, it gives a
criterion for each of the possibilities
$$
s(p^r)=r,\ r-1,\ r-2,\ldots,1,0.
$$ Using this criterion, we conclude that Ramassamy's preperiod conjecture is false and that the smallest counterexample is $5^5$.

We then use this criterion for computational searches and we further conjecture
$$
s(p^r)\ge r-2
$$
for every odd prime $p$ and every $r\ge2$, verified for $p^r\le10^{100000}$
\paragraph{Notation.}
For $\xi\in S_r$, we use a bar to denote its image in the quotient $S_r/pS_r$, that is,
$\bar\xi:=\xi+pS_r$. More generally, when working in $S_r/p^kS_r$, we use the same
bar notation for the image of $\xi$ in this quotient, with the ambient quotient ring always
clear from the context. For a commutative ring $A$ and for all $\alpha\in A$, we set $\alpha^0:=1_A$; in particular with the convention $0_A^{\,0}=1_A$.
\section{Eventually periodic sequences}\label{sec:periodic}
Let $a=(a_n)_{n\ge0}$ be a sequence in a set $X$. A positive integer $P$ is an
\emph{eventual period} of $a$ if there exists $t\in\mathbb Z_{\ge0}$ such that
$a_{n+P}=a_n$ for every $n\in\mathbb Z_{\ge0}$ with $n\ge t$, and we put
\[
\Pi_a:=\{P\in\Z_{\ge1}:\exists\,t\in\Z_{\ge0} \forall n\ge t,\ a_{n+P}=a_n\},
\]
so that $a$ is eventually periodic exactly when $\Pi_a\ne\varnothing$. In that case
$\Pi_a$ has a least element by well-ordering, and we set $d(a):=\min\Pi_a$, the
\emph{minimal eventual period}.

\begin{lemma}\label{lem:pi}
Let $a$ be eventually periodic and $d=d(a)$. Then $\Pi_a=\{kd:k\ge1\}$.
\end{lemma}

\begin{proof}
If $P,Q\in\Pi_a$ hold from indices $t_P,t_Q$ on and $t=\max\{t_P,t_Q\}$, then for $n\ge t$
we have $a_{n+(P+Q)}=a_{(n+Q)+P}=a_{n+Q}=a_n$, so $P+Q\in\Pi_a$; iterating gives
$kP\in\Pi_a$ for all $k\ge1$. If moreover $P>Q$, then $a_{n+(P-Q)}=a_{n+P}=a_n$ for
$n\ge t$, because $n+(P-Q)\ge t_Q$; so $P-Q\in\Pi_a$. Now let $P\in\Pi_a$ and write
$P=kd+\rho$ with $0\le\rho<d$; here $k\ge1$ since $d\le P$. If $\rho>0$ then
$\rho=P-kd\in\Pi_a$ by the two closure properties, contradicting the minimality of $d$.
Hence $d\mid P$, and the converse inclusion is the closure under multiples.
\end{proof}

For $P\in\Pi_a$ put $\mathcal T_{a,P}:=\{t\in\mathbb Z_{\geq0}:\forall n\in\mathbb Z_{\geq0},\ n\ge t\Longrightarrow a_{n+P}=a_n\}$, the set of admissible thresholds for $P$.

\begin{lemma}\label{lem:threshold}
Let $a$ be eventually periodic and $d=d(a)$. Then $\mathcal T_{a,P}=\mathcal T_{a,d}$ for
every $P\in\Pi_a$.
\end{lemma}

\begin{proof}
Write $P=kd$ (Lemma~\ref{lem:pi}). If $t\in\mathcal T_{a,d}$ and $n\ge t$, then
$n+jd\ge t$ for $j\ge0$, so $a_{n+P}=a_{n+kd}=\cdots=a_{n+d}=a_n$; thus
$\mathcal T_{a,d}\subseteq\mathcal T_{a,P}$.

Conversely let $t\in\mathcal T_{a,P}$, and fix some $u\in\mathcal T_{a,d}$, which exists
because $d\in\Pi_a$. Given $n\ge t$, set
$\ell:=\max\{0,\lceil (u-n)/P\rceil\}$, so that $\ell\in\mathbb Z_{\ge0}$ and
$n+\ell P\ge u$. Since $n+jP\ge t$ and $n+d+jP\ge t$ for every $j\ge0$, applying the
period $P$ repeatedly gives $a_{n+\ell P}=a_n$ and
$a_{n+d+\ell P}=a_{n+d}$; and $a_{n+d+\ell P}=a_{n+\ell P}$ because
$n+\ell P\ge u$. Combining the three equalities yields $a_{n+d}=a_n$, so
$t\in\mathcal T_{a,d}$.
\end{proof}

Consequently $\mathcal T_a:=\mathcal T_{a,P}$ does not depend on $P\in\Pi_a$, and we may
define the \emph{preperiod} $s(a):=\min\mathcal T_a$.

\begin{corollary}\label{cor:preperiod}
Let $a$ be eventually periodic, $s=s(a)$, $P\in\Pi_a$ and $t\ge0$. Then
$a_{n+P}=a_n$ for all $n\ge t$ if and only if $s\le t$.
\end{corollary}

\begin{proof}
The hypothesis says $t\in\mathcal T_a$, whence $s\le t$; conversely $s\in\mathcal T_a$
gives $a_{n+P}=a_n$ for $n\ge s$, hence for $n\ge t$ whenever $t\ge s$.
\end{proof}

\section{The ring \texorpdfstring{$S_r$}{S\_r}}\label{sec:ring}
In this section, we construct the algebraic background in which we make extensive use of an element \(i\) satisfying \(i^2=-1\), since identities involving \(i\) will appear repeatedly in the next sections. Since our problem concerns the Euler up/down numbers modulo \(p^r\), it is clear that the underlying ring should at least retain the arithmetic of integers modulo \(p^r\). Thus, \(\mathbb Z/p^r\mathbb Z\) is chosen naturally for this construction.

At first, one could choose such a ring separately according to \(p\bmod 4\):
\[
S'_r:=
\begin{cases}
\mathbb Z/p^r\mathbb Z,
& p\equiv1\pmod4,\\[4pt]
(\mathbb Z/p^r\mathbb Z)[x]/(x^2+1),
& p\equiv3\pmod4.
\end{cases}
\]

When \(p\equiv1\pmod4\), the equation \(x^2+1=0\) already has a solution modulo \(p\)~\cite[p.~52]{IrelandRosen}, and such a solution can be lifted modulo \(p^r\) by Hensel's lemma~\cite[p.~89]{Gouvea}. Thus, in this case, an element \(i\) satisfying \(i^2=-1\) can already be chosen inside \(\mathbb Z/p^r\mathbb Z\), so adjoining a new element is unnecessary.

When \(p\equiv3\pmod4\), on the other hand, such an element \(i\) does not exist in the ring \(\mathbb Z/p^r\mathbb Z\); we may therefore obtain such an element simply by passing to \((\mathbb Z/p^r\mathbb Z)[x]/(x^2+1)\) and taking \(i\) to be the class of \(x\), that is, \(i=[x]\).

However, we give up this minimality and work instead in the single ring
\[
S_r:=(\mathbb Z/p^r\mathbb Z)[x]/(x^2+1)
\]
for every odd prime \(p\). The main reason for this choice is not algebraic necessity, but uniformity of the later arguments. Treating the cases \(p\equiv1\pmod4\) and \(p\equiv3\pmod4\) separately throughout the paper would introduce an additional distinction that is not essential to the main ideas and could unnecessarily complicate the exposition. Moreover, working in the ring \(S_r\) avoids this repeated case separation and constructively guarantees the existence of an element \(i\) satisfying \(i^2=-1\). The idea is analogous to the standard construction of the complex numbers \(\mathbb C=\mathbb R[x]/(x^2+1)\).Now, we present the following lemma

\begin{lemma}\label{lem:Sr}
Let $p$ be an odd prime and $r\ge1$, and put $A_r:=\mathbb Z/p^r\mathbb Z$. Let
$\iota_r:A_r\to S_r$ be the natural map. Then:
\begin{enumerate}
\item[\textup{(i)}] every $\xi\in S_r$ can be written uniquely as
$\xi=\iota_r(a)+\iota_r(b)i$ with $a,b\in A_r$; equivalently, if
$\iota_r(a)+\iota_r(b)i=\iota_r(c)+\iota_r(d)i$, then $a=c$ and $b=d$;
\item[\textup{(ii)}] $\iota_r$ is an injective unital ring homomorphism carrying units
to units; in particular, if $m\in\Z$ and $p\nmid m$, then $[m]_{p^r}$ maps to a unit
of $S_r$;
\item[\textup{(iii)}] $i^2=-1$, $i^{-1}=-i$ and $\ord(i)=4$;
\item[\textup{(iv)}] $2$, $1\pm i$ and $-1\pm i$ are units of $S_r$;
\item[\textup{(v)}] for $\xi\in S_r$ and $0\le k\le r-1$, one has
$p^{r-k}\xi=0$ if and only if $\xi\in p^kS_r$;
\item[\textup{(vi)}] $p^nS_r=\{0\}$ for every $n\ge r$.
\end{enumerate}
\end{lemma}

\begin{proof}
Since $x^2+1$ is monic of degree $2$, every class in $A_r[x]/(x^2+1)$ has a unique
representative of degree $<2$ \cite[pp.~147--150]{Aluffi}, which gives \textup{(i)}.
For \textup{(ii)}, $\iota_r$ is a unital ring homomorphism, and if $\iota_r(a)=0$ then
$a=0$ by the uniqueness in \textup{(i)}, so $\iota_r$ is injective. Unital
homomorphisms preserve units, and $[m]_{p^r}\in A_r^\times$ whenever $p\nmid m$, so
its image under $\iota_r$ is a unit of $S_r$.

We identify $A_r$ with its image in $S_r$. Thus,
 the representation in \textup{(i)} is written simply as $\xi=a+bi$, by convention.

For \textup{(iii)}, $i^2=[x^2]=-1$ by construction, so $i(-i)=1$ and $i^4=1$.
If $1=-1$ in $S_r$, then by the injectivity of $\iota_r$ we would have $1=-1$ in
$A_r$, which is impossible. Hence $i^2=-1\ne1$. Moreover, $i=1$ would force
$i^2=1$, while $i^3=1$ would give $i^4=i$, hence $i=1$, again a contradiction.
Therefore $\ord(i)=4$.

For \textup{(iv)}, since $p\nmid2$, we have $2\in S_r^\times$ by \textup{(ii)}, and
$(1+i)(1-i)=2$ and $(-1+i)(-1-i)=2$, so $1\pm i$ and $-1\pm i$ are units of $S_r$.

For \textup{(v)}, write $\xi=a+bi$ with $a=[u]_{p^r}$ and $b=[v]_{p^r}$. By
\textup{(i)}, $p^{r-k}\xi=0$ is equivalent to $p^{r-k}a=p^{r-k}b=0$ in $A_r$,
i.e.\ to $p^r\mid p^{r-k}u$ and $p^r\mid p^{r-k}v$, hence to $p^k\mid u$ and
$p^k\mid v$. Thus $u=p^kc$ and $v=p^kd$ for some $c,d\in\mathbb Z$, so
$a=p^k[c]_{p^r}$ and $b=p^k[d]_{p^r}$, whence
$\xi=p^k([c]_{p^r}+[d]_{p^r}i)\in p^kS_r$. Conversely, if $\xi\in p^kS_r$, write $\xi=p^k\eta$ for some $\eta\in S_r$. Then
$p^{r-k}\xi=p^r\eta=0$, since $p^r=0$ in $S_r$. For
\textup{(vi)}, if $n\ge r$, then $p^n=p^{n-r}p^r=0$ in $S_r$, and hence
$p^nS_r=\{0\}$.
\end{proof}
\begin{corollary}\label{cor:units}
Let $a,b\in\Z$ with $p\nmid(a-b)$. Then $a-b$ and $(a-b)i$ are units of $S_r$; this
applies in particular to distinct $a,b\in\{1,\dots,p-1\}$. Moreover the image of a unit
of $S_r$ in $S_r/pS_r$ is a unit, so for distinct $a,b\in\{1,\dots,p-1\}$ the element
$\overline{(a-b)i}$ is a unit of $S_r/pS_r$.
\end{corollary}

\begin{proof}
If $p\nmid(a-b)$ then $[a-b]_{p^r}\in A_r^\times$, so $a-b\in S_r^\times$ by
Lemma~\ref{lem:Sr}(ii), and $i\in S_r^\times$ by (iii) above; hence
$(a-b)i\in S_r^\times$. For distinct $a,b\in\{1,\dots,p-1\}$ we have
$-p<a-b<p$ and $a-b\ne0$, hence $p\nmid(a-b)$. Finally,
$\overline{(a-b)i}$ is a unit of $S_r/pS_r$, since quotient homomorphisms preserve units.
\end{proof}

\section{Hurwitz series and the frequency expansion}\label{sec:hurwitz}

\subsection*{The Hurwitz ring}

Let $A$ be a commutative ring with identity. A \emph{Hurwitz series} over $A$ is a
function $a:\Z_{\ge0}\to A$, written as a sequence $(a_0,a_1,\dots)$ or formally as
$\sum_{n\ge0}a_nz^n/n!$, where the symbols $z^n/n!$ are formal and no division by $n!$
takes place in $A$. The set $A\langle z\rangle$ of all such series is a commutative ring
under coefficientwise addition and the Hurwitz product \cite[p.~1846]{Keigher}
\[
(a*b)_n:=\sum_{k=0}^{n}\binom{n}{k}a_kb_{n-k},
\]
the binomial coefficients being read in $A$ via $\Z\to A$; its identity is
$1_{A\langle z\rangle}=(1_A,0_A,0_A,\dots)$. We write $A[[z]]$ for the ordinary formal
power series ring, with coefficientwise addition and Cauchy product
$(ab)_n=\sum_{k=0}^{n}a_kb_{n-k}$.
\begin{lemma}\label{lem:hurwitz-basic}
Let $A$ and $B$ be commutative rings with identity. Then:
\begin{enumerate}
\item[\textup{(i)}] if $\varphi:A\to B$ is a unital ring homomorphism, then applying
$\varphi$ coefficientwise defines unital ring homomorphisms
$\varphi_{[z]}:A[z]\to B[z]$ and
$\varphi_{\langle z\rangle}:A\langle z\rangle\to B\langle z\rangle$;
\item[\textup{(ii)}] $f\in A\langle z\rangle$ is a unit if and only if
$f_0\in A^\times$;
\item[\textup{(iii)}] $f\in A[[z]]$ is a unit if and only if $f_0\in A^\times$;
\item[\textup{(iv)}] the map
$\psi_A:A[[z]]\to A\langle z\rangle$, defined by
$\psi_A(h)(n):=n!_A\,h(n)$, is a unital ring homomorphism.
\end{enumerate}
\end{lemma}

\begin{proof}
For \textup{(i)}, in both cases the $n$-th coefficient of the sum is $a_n+b_n$. The $n$-th coefficient of the product is $\sum_{k=0}^{n}a_kb_{n-k}$ in the polynomial case and $\sum_{k=0}^{n}\binom{n}{k}a_kb_{n-k}$ in the Hurwitz case; all these expressions are preserved by $\varphi$. For \textup{(ii)}, if $f*g=1$, then $f_0g_0=1$, so $f_0\in A^\times$; conversely, if $f_0\in A^\times$, the inverse is determined recursively by $g_0=f_0^{-1}$ and $g_n=-f_0^{-1}\sum_{k=1}^{n}\binom{n}{k}f_kg_{n-k}$ for $n\ge1$.
The proof of \textup{(iii)} is the same with the Cauchy product, giving
$g_n=-f_0^{-1}\sum_{k=1}^{n}f_kg_{n-k}$. For \textup{(iv)}, additivity and unitality
are immediate, while multiplicativity follows from
$n!=\binom{n}{k}k!(n-k)!$, which converts the Cauchy product into the Hurwitz product.
\end{proof}

For $\alpha\in A$ define the \emph{formal exponential}
\[
\varepsilon(\alpha):=(1_A,\alpha,\alpha^2,\alpha^3,\dots)
=\sum_{n\ge0}\alpha^n\frac{z^n}{n!}\in A\langle z\rangle.
\]
The binomial theorem applied to the Hurwitz product gives
$\varepsilon(\alpha)\varepsilon(\beta)=\varepsilon(\alpha+\beta)$; hence
$\varepsilon(\alpha)^k=\varepsilon(k\alpha)$ for $k\ge0$, and
$\varepsilon(\alpha)\in A\langle z\rangle^\times$ with
$\varepsilon(\alpha)^{-1}=\varepsilon(-\alpha)$, since
$\varepsilon(0_A)=1_{A\langle z\rangle}$.
\subsection*{The Euler series over \texorpdfstring{$S_r$}{S\_r}}

In $\Q[[z]]$ put $\cos z:=\sum_{m\ge0}(-1)^mz^{2m}/(2m)!$ and
$\sin z:=\sum_{m\ge0}(-1)^mz^{2m+1}/(2m+1)!$. The constant coefficient of $\cos z$ is
$1$, so $\cos z$ is a unit by Lemma~\ref{lem:hurwitz-basic}\textup{(iii)}, and we may set
$\sec z:=(\cos z)^{-1}$ and $\tan z:=(\sin z)(\cos z)^{-1}$. Writing
$\epsilon:=\sum_{n\ge0}E_nz^n/n!$, Andr\'e's identity \eqref{eq:andre} reads
$\epsilon=\sec z+\tan z$, that is,
\[
\epsilon\cos z=1+\sin z\qquad\text{in }\Q[[z]].
\]
Applying $\psi_\Q$ from Lemma~\ref{lem:hurwitz-basic}\textup{(iv)} transports this equation to
$\Q\langle z\rangle$, where the Hurwitz coefficients are
\[
\cos z=(1,0,-1,0,1,0,-1,\dots),
\quad
\sin z=(0,1,0,-1,0,1,0,\dots),
\quad
\epsilon=(E_0,E_1,E_2,\dots).
\]
All coefficients are integers, so the same identity holds in
$\Z\langle z\rangle$. Let $\varphi:\Z\to S_r$ be the canonical homomorphism. By
Lemma~\ref{lem:hurwitz-basic}\textup{(i)}, applying $\varphi$ coefficientwise gives the
unital ring homomorphism $ \varphi_{\langle z\rangle}:\Z\langle z\rangle\to S_r\langle z\rangle=\Hur.$

Writing
\[
F:=\varphi_{\langle z\rangle}(\epsilon)
=\bigl((E_n)_{S_r}\bigr)_{n\ge0}\in\Hur,
\]
and using the same symbols $\cos z$ and $\sin z$ for their images under
$\varphi_{\langle z\rangle}$, we obtain
\begin{equation}\label{eq:Fcos}
F\cos z=1+\sin z\qquad\text{in }\Hur.
\end{equation}
We identify each $\xi\in S_r$ with the constant series $(\xi,0,0,\dots)\in\Hur$.
\begin{lemma}\label{lem:u}
Let $N:=p^r$ and $u:=\varepsilon(i)\in\Hur$. Then:
\begin{enumerate}
\item[\textup{(i)}] $u\in\Hur^\times$ and $u^{-1}=\varepsilon(-i)$;
\item[\textup{(ii)}] $u^k=\varepsilon(ki)$ and $u^k(n)=(ki)^n$ for all $k,n\ge0$;
\item[\textup{(iii)}] $u^N=1_\Hur$;
\item[\textup{(iv)}]
\[
\cos z=\frac{u+u^{-1}}{2},
\qquad
\sin z=\frac{u-u^{-1}}{2i}.
\]
\end{enumerate}
\end{lemma}

\begin{proof}
Parts \textup{(i)} and \textup{(ii)} follow from the properties of $\varepsilon$ recorded above. For
\textup{(iii)}, since $N=p^r$, Lemma~\ref{lem:Sr}\textup{(vi)} gives $N=0$ in $S_r$,
so $Ni=0$ and
$u^N=\varepsilon(Ni)=\varepsilon(0)=1_\Hur$.

For \textup{(iv)}, coefficientwise
$(u\pm u^{-1})(n)=i^n\pm(-i)^n$, which equals $2(-1)^m$ and $0$, respectively, for
$n=2m$, and $0$ and $2i(-1)^m$, respectively, for $n=2m+1$. Hence
$u+u^{-1}=2\cos z$ and $u-u^{-1}=2i\sin z$. By
Lemma~\ref{lem:Sr}\textup{(iii),(iv)}, $2,2i\in S_r^\times$, and therefore by
Lemma~\ref{lem:hurwitz-basic}\textup{(ii)} they are units of $\Hur$, giving the stated
identities.
\end{proof}

Substituting the identities for sine and cosine in Lemma~\ref{lem:u}\textup{(iv)} into \eqref{eq:Fcos} and multiplying by
$2iu$ gives
$Fi(u^2+1)=u^2+2iu-1$. Since
$u^2+2iu-1=(u+i)^2$ and $i(u^2+1)=(u+i)(1+iu)$, we obtain
\[
F(u+i)(1+iu)=(u+i)^2.
\]
Both $u+i$ and $1+iu$ have constant coefficient $1+i\in S_r^\times$ by
Lemma~\ref{lem:Sr}\textup{(iv)}, hence are units of $\Hur$ by
Lemma~\ref{lem:hurwitz-basic}\textup{(ii)}. Cancelling $u+i$ and inverting $1+iu$
therefore yields
\begin{equation}\label{eq:Fclosed}
F=(u+i)(1+iu)^{-1}.
\end{equation}

\begin{remark}\label{rem:fourier}
The identities in Lemma~\ref{lem:u}\textup{(iv)} for $\cos z$ and $\sin z$ are formally analogous to the classical formulas $\cos z=(e^{iz}+e^{-iz})/2$ and $\sin z=(e^{iz}-e^{-iz})/(2i)$ in complex analysis. Furthermore, since the powers $u^k$ play the formal role of the exponential modes $e^{ikz}$,the motivation for calling the next proposition \emph{Frequency expansion} comes from this analogy, with the terminology being purely algebraic.
\end{remark}
\begin{proposition}[Frequency expansion]\label{prop:expansion}
Let $p$ be an odd prime, $r\ge1$, $N:=p^r$, $u:=\varepsilon(i)\in\Hur$, and put
$\eta:=i(i^{-N}-1)$. Then:
\begin{enumerate}
\item[\textup{(i)}] $\eta=1-i$ if $N\equiv1\pmod4$ and $\eta=-1-i$ if $N\equiv3\pmod4$;
in particular $\eta\in S_r^\times$;
\item[\textup{(ii)}] setting $H_0:=(1+i^{-N})\eta^{-1}$ and $H_k:=2i^{-k}\eta^{-1}$ for
$1\le k\le N-1$, one has $H_0=1$ if $N\equiv1\pmod4$, $H_0=-1$ if $N\equiv3\pmod4$, and
$H_k\in S_r^\times$ for all $0\le k\le N-1$;
\item[\textup{(iii)}] $\displaystyle F=\sum_{k=0}^{N-1}H_ku^k$ in $\Hur$;
\item[\textup{(iv)}] $\displaystyle F_n=\sum_{\substack{1\le k\le N-1\\ p\nmid k}}H_k(ki)^n$
for every $n\ge r$.
\end{enumerate}
\end{proposition}

\begin{proof}
$N=p^r$ is odd. If $N\equiv1\pmod4$ then $i^{-N}=i^{-1}=-i$ and $\eta=i(-i-1)=1-i$; if
$N\equiv3\pmod4$ then $i^{-N}=i$ and $\eta=i(i-1)=-1-i$. Both are units by
Lemma~\ref{lem:Sr}(iv), giving (i). The two cases give
$H_0=(1-i)/(1-i)=1$ and $H_0=(1+i)/(-1-i)=-1$ respectively, and for $1\le k\le N-1$ the
element $H_k=2i^{-k}\eta^{-1}$ is a product of units by Lemma~\ref{lem:Sr}(iii),(iv);
this is (ii).

Put $G:=\sum_{k=0}^{N-1}H_ku^k$. Using $u^N=1$ by Lemma~\ref{lem:u}\textup{(iii)},
\[
(1+iu)G
=\bigl(H_0+iH_{N-1}\bigr)+\bigl(H_1+iH_0\bigr)u
+\sum_{k=2}^{N-1}\bigl(H_k+iH_{k-1}\bigr)u^k .
\]
For $2\le k\le N-1$ we have $i^{2-k}=-i^{-k}$, so
$H_k+iH_{k-1}=\tfrac2\eta(i^{-k}+i^{2-k})=0$. Next, since
$\eta=i^{1-N}-i$,
\[
H_1+iH_0=\frac{2i^{-1}+i(1+i^{-N})}{\eta}=\frac{i^{1-N}-i}{\eta}=1,
\]
and, using $i^{2-N}=-i^{-N}$ together with $i\eta=i^2(i^{-N}-1)=1-i^{-N}$,
\[
H_0+iH_{N-1}=\frac{1+i^{-N}+2i^{2-N}}{\eta}=\frac{1-i^{-N}}{\eta}=i .
\]
Hence $(1+iu)G=u+i$ ; since $1+iu\in\Hur^\times$ , we obtain
$G=(u+i)(1+iu)^{-1}=F$ by \eqref{eq:Fclosed}, which proves \textup{(iii)}.

Taking $n$-th coefficients and using $u^k(n)=(ki)^n$ by  Lemma~\ref{lem:u}\textup{(ii)} gives
$F_n=\sum_{k=0}^{N-1}H_k(ki)^n$. If
$n\ge r$ and $p\mid k$, say $k=p\ell$, then $(ki)^n=p^n(\ell i)^n=0$ by
Lemma~\ref{lem:Sr}(vi), leaving (iv).
\end{proof}

\begin{remark}\label{rem:transfer}
By Lemma~\ref{lem:Sr}(ii) the map $\iota_r:A_r\to S_r$ is injective, and
$F_n=\iota_r([E_n]_{p^r})$; hence $F_m=F_n$ if and only if $E_m\equiv E_n\pmod{p^r}$. The
sequence $(F_n)_{n\ge0}$ therefore has exactly the same eventual periods, minimal
eventual period and preperiod as $(E_n\bmod p^r)_{n\ge0}$, meaning that we can work with $F_n$ for our analysis.
\end{remark}

\begin{remark}
The coefficients $H_k$ arise naturally from \eqref{eq:Fclosed}. Since
$u+i=i\,u^0+u$ is already a finite $S_r$-linear combination of powers of
$u$, it remains to express $(1+iu)^{-1}$ in the same form. Expanding
$(1+iu)^{-1}$ in this way, multiplying by $u+i$, and grouping the resulting
terms according to the powers of $u$ yields 
$H_0,\ldots,H_{N-1}$.
\end{remark}
\section{The period conjecture}\label{sec:period}

\begin{lemma}\label{lem:vdm}
Let $A$ be a commutative ring with identity, $t\ge1$, and let
$\lambda_1,\dots,\lambda_t\in A^\times$ satisfy $\lambda_k-\lambda_j\in A^\times$ for all
$1\le j<k\le t$. If $c_1,\dots,c_t\in A$ satisfy
$\sum_{j=1}^{t}c_j\lambda_j^{\,n}=0$ for $t$ consecutive values
$n=n_0,\dots,n_0+t-1$, then $c_1=\cdots=c_t=0$.
\end{lemma}

\begin{proof}
The hypotheses say $VD\,(c_1,\dots,c_t)^{\mathsf T}=0$, where $V$ is the Vandermonde
matrix on $\lambda_1,\dots,\lambda_t$ and
$D=\operatorname{diag}(\lambda_1^{n_0},\dots,\lambda_t^{n_0})$. Then
$\det(VD)=\bigl(\prod_j\lambda_j^{n_0}\bigr)\prod_{j<k}(\lambda_k-\lambda_j)$ is a product
of units, hence a unit, so $VD$ is invertible over $A$. Hence \(c_1=\cdots=c_t=0\).
\end{proof}

\begin{remark}
Since $d(a)$ and $s(a)$ denote the minimal eventual period and preperiod of a
sequence $a$, we shall, consistently with the notation in \eqref{eq:kb} and \eqref{eq:conj}, abbreviate
\[
d\bigl((E_n\bmod p^r)_{n\ge0}\bigr)=d(p^r),
\qquad
s\bigl((E_n\bmod p^r)_{n\ge0}\bigr)=s(p^r).
\]
\end{remark}
\begin{lemma}\label{lem:upper}
Let $p$ be an odd prime, $r\ge1$, and $d_1:=\operatorname{lcm}(p-1,4)$. Then
$(ki)^{p^{r-1}d_1}=1$ in $S_r$ for every $k\in\Z$ with $p\nmid k$; consequently
$F_{n+p^{r-1}d_1}=F_n$ for all $n\ge r$, and therefore $d(p^r)\mid p^{r-1}d_1$ and
$s(p^r)\le r$.
\end{lemma}

\begin{proof}
Put $\delta:=d_1/(p-1)\in\{1,2\}$. For $p\nmid k$ Euler's theorem
\cite[p.~33]{IrelandRosen} gives $k^{p^{r-1}(p-1)}=1$ in $A_r$, hence in $S_r$ by
Lemma~\ref{lem:Sr}(ii), so $k^{p^{r-1}d_1}=\bigl(k^{p^{r-1}(p-1)}\bigr)^{\delta}=1$; and
$i^{p^{r-1}d_1}=1$ since $4\mid d_1$. Thus $(ki)^{p^{r-1}d_1}=1$. For $n\ge r$
Proposition~\ref{prop:expansion}(iv) expresses $F_n$ as a sum over $p\nmid k$ of terms
$H_k(ki)^n$, each of which is unchanged when $n$ is increased by $p^{r-1}d_1$; hence
$F_{n+p^{r-1}d_1}=F_n$. So $p^{r-1}d_1$ is an eventual period with threshold $r$, and
Lemma~\ref{lem:pi} and Corollary~\ref{cor:preperiod} give the two conclusions.
\end{proof}

\begin{lemma}\label{lem:lower}
Let $p$ be an odd prime and $d_1:=\operatorname{lcm}(p-1,4)$. Then $d_1\mid d(p)$.
\end{lemma}

\begin{proof}
Work at $r=1$ and put $P:=d(p)$. Choose $t\ge0$ such that $F_{n+P}=F_n$ for all
$n\ge t$, and set $n_0:=\max\{t,1\}$. Since $r=1$, Proposition~\ref{prop:expansion}(iv)
applies for $n\ge1$, with $p\nmid k$ automatic for $1\le k\le p-1$. Thus, for all $n\ge n_0$,
\[
0=F_{n+P}-F_n=\sum_{k=1}^{p-1}H_k\bigl((ki)^P-1\bigr)(ki)^n .
\]
Set $\lambda_k:=ki$ and $c_k:=H_k((ki)^P-1)$. Each $\lambda_k$ is a unit of $S_1$ by
Lemma~\ref{lem:Sr}\textup{(ii),(iii)}, and
for distinct $j,k\in\{1,\dots,p-1\}$ the difference $\lambda_k-\lambda_j=(k-j)i$ is a
unit by Corollary~\ref{cor:units}. Applying Lemma~\ref{lem:vdm} to the $p-1$ consecutive
values $n=n_0,\dots,n_0+p-2$ gives $c_k=0$ for $1\le k\le p-1$, and since
$H_k\in S_1^\times$ we conclude $(ki)^P=1$ for $1\le k\le p-1$.

Taking $k=1$ gives $i^P=1$, hence $4\mid P$ by Lemma~\ref{lem:Sr}(iii). For general $k$
we then get $(ki)^P=k^P=1$ in $S_1$, hence in $A_1$ by injectivity. As $A_1$ is a
field, the classes $1,\dots,p-1$ are precisely the elements of $A_1^\times$, which is
cyclic of order $p-1$ \cite[p.~40]{IrelandRosen}; applying $x^P=1$ to a generator gives
$p-1\mid P$. This and above $4\mid P$ gives $d_1=\operatorname{lcm}(p-1,4)\mid P=d(p)$.
\end{proof}
\begin{theorem}[Knuth--Buckholtz \cite{KB}]\label{thm:kb}
Let $p$ be an odd prime. Then $d(p)=\operatorname{lcm}(p-1,4)$, and for every $r\ge1$,
\[
s(p^r)\le r,\qquad d(p^r)\mid p^{r-1}d(p).
\]
\end{theorem}

\begin{proof}
Lemma~\ref{lem:upper} at $r=1$ gives $d(p)\mid d_1$ and Lemma~\ref{lem:lower} gives
$d_1\mid d(p)$, so $d(p)=d_1$. The remaining assertions are Lemma~\ref{lem:upper} with
$d_1=d(p)$.
\end{proof}

We include this proof because it falls out of the frequency expansion; it is logically
independent of the original argument of \cite{KB}.

Since \(d(p)=\operatorname{lcm}(p-1,4)\mid 4(p-1)\), while \(p\nmid4\) and
\(p\nmid p-1\), we have \(p\nmid d(p)\). Hence, if $b_1=1<b_2<\cdots <b_m=d(p)$ are the positive
divisors of $d(p)$, every positive divisor of $p^{r-1}d(p)$, and therefore $d(p^r)$,
is of the form $p^jb_\ell$ with $0\le j\le r-1$ and $1\le\ell\le m$. Thus the possible
values of $d(p^r)$ lie in
\[
\begin{pmatrix}
1 & b_2 & \cdots & b_{m-1} & d(p)\\
p & pb_2 & \cdots & pb_{m-1} & pd(p)\\
p^2 & p^2b_2 & \cdots & p^2b_{m-1} & p^2d(p)\\
\vdots & \vdots & & \vdots & \vdots\\
p^{r-1} & p^{r-1}b_2 & \cdots & p^{r-1}b_{m-1} & p^{r-1}d(p)
\end{pmatrix}.
\]
We shall next prove $d(p)\mid d(p^r)$, which forces $d(p^r)$ to lie in the last column.

\begin{lemma}\label{lem:divides}
Let $p$ be an odd prime and $r\ge1$. Then $d(p)\mid d(p^r)$, and consequently
\[
d(p^r)\in\bigl\{d(p),\,pd(p),\,p^2d(p),\dots,p^{r-1}d(p)\bigr\}.
\]
\end{lemma}

\begin{proof}
Eventually $p^r\mid E_{n+d(p^r)}-E_n$, hence $p\mid E_{n+d(p^r)}-E_n$; so $d(p^r)$ is an
eventual period modulo $p$ and Lemma~\ref{lem:pi} gives $d(p)\mid d(p^r)$.
\end{proof}
For $r=1$, the Vandermonde argument of Lemma~\ref{lem:lower} works because the frequency indices are $1,\dots,p-1$, so distinct indices never differ by a multiple of $p$. For $r>1$, however, the range $1\le k\le p^r-1$ already contains, for instance, $1$ and $1+p$, and the corresponding frequencies differ by $(1+p)i-i=pi$, which is not a unit of $S_r$. Thus, we cannot rely on the same argument in $S_r$.

On the other hand, we already know $d(p^r)\in\{d(p),pd(p),\dots,p^{r-1}d(p)\}$, so it is enough to test the penultimate candidate $L:=p^{r-2}d(p)$. If $L$ is an eventual period, then $d(p^r)\mid L$ and the conjectured value is impossible; if it is not, none of the smaller candidates can occur since each divides $L$. Thus the problem reduces to deciding whether $L$ is an eventual period.

For $n\ge r$, Proposition~\ref{prop:expansion}\textup{(iv)} and $i^L=1$ since $4\mid L$, give $F_{n+L}-F_n=\sum_{p\nmid k}H_k(k^L-1)(ki)^n$. The dependence on $L$ is therefore concentrated in the factors $k^L-1$. We shall show that each $k^L-1$ terms in the sum contains the factor $p^{r-1}$, so that $F_{n+L}-F_n=p^{r-1}X_n$ for a suitable $X_n\in S_r$. Lemma~\ref{lem:Sr}\textup{(v)} then yields $F_{n+L}-F_n=0\iff X_n\in pS_r\iff\bar X_n=\bar0$ in $S_r/pS_r$. Therefore, by considering $X_n$ in $S_r/pS_r$, that is, $\bar X_n$,
we can constructively determine whether $L$ is an eventual period.

Moreover, passing from $S_r$ to $S_r/pS_r$ makes the frequencies indexed by $k,k+p,k+2p,\dots$ collapse to the same residue class modulo $p$ in $S_r/pS_r$. Hence the original family is grouped into the $p-1$ nonzero residue classes modulo $p$, whose corresponding frequencies again have pairwise unit differences. As a result, passing to $\bar X_n$ in $S_r/pS_r$ simultaneously produces a criterion for whether $L$ is an eventual period and enables us to re-use the Vandermonde argument for $r>1$.

\begin{lemma}\label{lem:bn}
Let $p$ be an odd prime and $r\ge2$. Put
\[
\delta:=\frac{d(p)}{p-1}
=\begin{cases}1,&p\equiv1\pmod4,\\2,&p\equiv3\pmod4,\end{cases}
\qquad
M:=p^{r-2}(p-1),\qquad L:=\delta M=p^{r-2}d(p).
\]
For $1\le k\le p^r-1$ with $p\nmid k$, put
\[
Q_k:=\frac{k^M-1}{p^{r-1}}\in\Z.
\]
Then:
\begin{enumerate}
\item[\textup{(i)}] $(ki)^L-1=\delta p^{r-1}Q_k$ in $S_r$ for every such $k$;
\item[\textup{(ii)}] for $n\ge r$, setting
\[
X_n:=\delta\!\!\sum_{\substack{1\le k\le p^r-1\\p\nmid k}}
H_kQ_k(ki)^n\in S_r,
\]
one has $F_{n+L}-F_n=p^{r-1}X_n$ and
\[
\bar X_n=\bar\delta\sum_{a=1}^{p-1}G_a(\bar a\,\bar i)^n,
\qquad
G_a:=\sum_{j=0}^{p^{r-1}-1}
\bar H_{a+jp}\,\bar Q_{a+jp}\in S_r/pS_r;
\]
\item[\textup{(iii)}] $L$ is an eventual period of $(F_n)_{n\ge0}$ if and only if
$G_1=\cdots=G_{p-1}=\bar0$.
\end{enumerate}
\end{lemma}

\begin{proof}
For $p\nmid k$, Euler's theorem \cite[p.~33]{IrelandRosen} gives
$k^{p^{r-2}(p-1)}\equiv1\pmod{p^{r-1}}$, so $Q_k\in\Z$ is well defined and
$k^M=1+p^{r-1}Q_k$ in $\Z$. Taking this equality modulo $p^r$ and reading it in
$S_r$, we retain $k^M=1+p^{r-1}Q_k$. Moreover
$4\mid d(p)\mid L$, so $i^L=1$.Thus, since
$\delta\in\{1,2\}$,
\[
(ki)^L-1=k^L-1=(1+p^{r-1}Q_k)^\delta-1
=\delta p^{r-1}Q_k
\]
in \(S_r\), the last equality for \(\delta=2\) following from
\(p^{2(r-1)}S_r=\{0\}\) by Lemma~\ref{lem:Sr}\textup{(vi)}, since
\(2(r-1)\ge r\) for \(r\ge2\). This proves \textup{(i)}.
For $n\ge r$, Proposition~\ref{prop:expansion}\textup{(iv)} and \textup{(i)} give
\[
F_{n+L}-F_n
=\sum_{\substack{1\le k\le p^r-1\\p\nmid k}}
H_k\bigl((ki)^L-1\bigr)(ki)^n
=p^{r-1}X_n.
\]
Reducing $X_n$ modulo $p$, every admissible $k$ is uniquely $k=a+jp$ with
$1\le a\le p-1$ and $0\le j\le p^{r-1}-1$, and
$\overline{a+jp}=\bar a$; grouping by $a$ gives \textup{(ii)}.

For \textup{(iii)}, suppose first that $L$ is an eventual period. Choose $t\ge0$ with
$F_{n+L}=F_n$ for all $n\ge t$, and put $n_0:=\max\{t,r\}$. Then
$p^{r-1}X_n=0$ for all $n\ge n_0$, so Lemma~\ref{lem:Sr}\textup{(v)} gives
$X_n\in pS_r$, equivalently $\bar X_n=\bar0$ in $S_r/pS_r$. Hence
\[
\bar\delta\sum_{a=1}^{p-1}G_a(\bar a\,\bar i)^n=\bar0
\qquad(n\ge n_0).
\]
The frequencies $\bar a\,\bar i$ are units and have pairwise unit differences by
Corollary~\ref{cor:units}; applying Lemma~\ref{lem:vdm} to
$n=n_0,\dots,n_0+p-2$ gives $\bar\delta G_a=\bar0$ for $1\le a\le p-1$.
Since $\delta\in\{1,2\}$, we have $\delta\in S_r^\times$, trivially for
$\delta=1$ and by Lemma~\ref{lem:Sr}\textup{(iv)} for $\delta=2$. Hence $\bar\delta$ is a unit of
$S_r/pS_r$, since quotient homomorphisms preserve units. Therefore
$G_a=\bar0$ for $1\le a\le p-1$.

Conversely, if $G_a=\bar0$ for $1\le a\le p-1$, then \textup{(ii)} gives
$\bar X_n=\bar0$ for every $n\ge r$, hence $X_n\in pS_r$. By
Lemma~\ref{lem:Sr}\textup{(v)}, $p^{r-1}X_n=0$, so
$F_{n+L}=F_n$ for every $n\ge r$. Thus $L$ is an eventual period.
\end{proof}
\begin{proposition}\label{prop:G1}
Let $p$ be an odd prime and $r\ge2$, and let $G_1\in S_r/pS_r$ be as in
Lemma~\ref{lem:bn}. Then $G_1=\bar1$.
\end{proposition}

\begin{proof}
By Lemma~\ref{lem:Sr}\textup{(iii),(iv)}, Proposition~\ref{prop:expansion}\textup{(i)}
and Corollary~\ref{cor:units}, the classes $\bar i$, $\bar2$ and $\bar\eta$ are units of
$S_r/pS_r$. Put
\[
\zeta:=\bar i^{\,-p},
\qquad
c_1:=\bar2\,\bar i^{\,-1}\bar\eta^{\,-1}.
\]
Since $p$ is odd, $\zeta\in\{\bar i,-\bar i\}$, so $\zeta^2=-\bar1$; moreover
$p^2\equiv1\pmod4$ gives
\[
\zeta^p=\bar i^{\,-p^2}=\bar i^{\,-1}=-\bar i.
\]
Thus $\zeta$, $\bar1-\zeta=\bar1\mp\bar i$ and
$\zeta^p-\bar1=-(\bar1+\bar i)$ are units of $S_r/pS_r$.
We first identify the two factors entering \(G_1\). Put
\(M=p^{r-2}(p-1)\), as in Lemma~\ref{lem:bn}. For
\(0\le j\le p^{r-1}-1\), expanding gives
\[
(1+jp)^M=1+Mjp+\sum_{h=2}^{M}\binom Mh j^hp^h.
\]
Every summand in the sum is divisible by \(p^r\): this is clear for
\(h\ge r\), while for \(2\le h<r\), writing \(h=p^au\) with \(p\nmid u\), we have
\(h-a\ge2\): indeed, this is immediate if \(a=0\), while for \(a\ge1\),
\(h=p^au\ge p^a\ge3^a\ge a+2\). Hence \(a\le h-2\le r-3\). The identity
\(h\binom Mh=M\binom{M-1}{h-1}\) therefore gives
\(p^{r-2-a}\mid\binom Mh\), and consequently each such summand is divisible by
\(p^{r-2-a+h}\), hence by \(p^r\). Thus
\[
(1+jp)^M\equiv1+Mjp
=1+jp^{r-1}(p-1)
\equiv1-jp^{r-1}\pmod{p^r},
\]
Thus there exists \(t\in\mathbb Z\) such that
\[
Q_{1+jp}
=\frac{(1+jp)^M-1}{p^{r-1}}
=-j+tp.
\]
Reducing this identity modulo $p^r$, viewing it in $S_r$, and then projecting onto
$S_r/pS_r$, we obtain
\[
\bar Q_{1+jp}=-\bar j.
\]
On the other hand, Proposition~\ref{prop:expansion}\textup{(ii)} gives
\[
\bar H_{1+jp}
=\bar2\,\bar i^{\,-1-jp}\bar\eta^{\,-1}
=c_1\zeta^j.
\]
Therefore
\[
G_1=-c_1\sum_{j=0}^{p^{r-1}-1}\bar j\,\zeta^j.
\]

Write $j=pm+s$ with $0\le m<p^{r-2}$ and $0\le s<p$. Then $\bar j=\bar s$, and the sum
factors as $RW$, where
\[
R:=\sum_{m=0}^{p^{r-2}-1}(\zeta^p)^m
=\frac{\zeta^{p^{r-1}}-\bar1}{\zeta^p-\bar1}
=\frac{\bar i^{\,-p^r}-\bar1}{\zeta^p-\bar1},
\qquad
W:=\sum_{s=0}^{p-1}\bar s\,\zeta^s.
\]
The identity $(1-X)^2\sum_{s=0}^{n}sX^s=X-(n+1)X^{n+1}+nX^{n+2}$ holds in $\mathbb Z[X]$. By Lemma~\ref{lem:hurwitz-basic}\textup{(i)}, the canonical map $\mathbb Z\to S_r/pS_r$ induces a coefficientwise unital ring homomorphism $\mathbb Z[X]\to (S_r/pS_r)[X]$. Evaluating at $X=\zeta$, taking $n=p-1$, and dividing by $(\bar1-\zeta)^2$, which is legitimate since $(\bar1-\zeta)^2$ is a unit in $S_r/pS_r$, together with $\bar p=\bar0$ and $\overline{p-1}=-\bar1$, gives
\[
W=\frac{\zeta-\zeta^{p+1}}{(\bar1-\zeta)^2}
=\frac{\zeta(\bar1-\zeta^p)}{(\bar1-\zeta)^2}.
\]
Recall Proposition~\ref{prop:expansion}, in $S_r/pS_r$ we also have
$\bar\eta=\bar i(\bar i^{\,-p^r}-\bar1)$, and therefore
\[
c_1R
=\bar2\,\bar i^{\,-1}\bar\eta^{\,-1}
 \frac{\bar\eta\,\bar i^{\,-1}}{\zeta^p-\bar1}
=\frac{-\bar2}{\zeta^p-\bar1}.
\]
Hence, noting
$(\bar1-\zeta)^2=-\bar2\zeta$,
\[
c_1RW
=\frac{-\bar2\,\zeta(\bar1-\zeta^p)}
{(\zeta^p-\bar1)(\bar1-\zeta)^2}
=\frac{\bar2\zeta}{-\bar2\zeta}
=-\bar1.
\]
Thus $G_1=-c_1RW=\bar1$.
\end{proof}
\begin{theorem}[Ramassamy's period conjecture]\label{thm:period}
Let $p$ be an odd prime and $r\ge1$. Then $d(p^r)=p^{r-1}d(p)$.
\end{theorem}

\begin{proof}
The case $r=1$ is immediate. Let $r\ge2$ and put $L:=p^{r-2}d(p)$. By
Lemmas~\ref{lem:divides} and~\ref{lem:pi}, together with
Lemma~\ref{lem:bn}\textup{(iii)},
\[
d(p^r)\ne p^{r-1}d(p)
\iff L\text{ is an eventual period}
\iff G_1=\cdots=G_{p-1}=\bar0.
\]
Thus one could test $G_1,G_2,\dots,G_{p-1}$ successively: if $G_1=\bar0$,
one would continue with $G_2$, and so on. However,
Proposition~\ref{prop:G1} already gives $G_1=\bar1\ne\bar0$, where
$\bar1\ne\bar0$ follows from Lemma~\ref{lem:Sr}\textup{(ii),(v)}. Hence logically
\[
d(p^r)=p^{r-1}d(p).
\]
\end{proof}

\section{Analysis of the preperiod}\label{sec:preperiod}
Let $p$ be an
odd prime and $r\ge1$, and put
\[
P:=p^{r-1}\operatorname{lcm}(p-1,4).
\]
By Lemma~\ref{lem:upper}, $P$ is an eventual period of $(F_n)_{n\ge0}$; we shall use
only this fact and do not rely on the fact that $P$ is the minimal eventual period, i.e.,
that $d(p^r)=P$ by Theorem~\ref{thm:period}.

Taking $n$-th coefficients in Proposition~\ref{prop:expansion}\textup{(iii)} and splitting
the frequency indices according to divisibility by $p$, for $n\ge0$ write
\[
F_n=U_n+B_n,
\qquad
U_n:=\sum_{\substack{1\le k\le p^r-1\\p\nmid k}}H_k(ki)^n,
\qquad
B_n:=\sum_{\substack{0\le k\le p^r-1\\p\mid k}}H_k(ki)^n .
\]

\begin{lemma}\label{lem:Bcriterion}
Let $p$ be an odd prime and $r\ge1$. For $1\le k\le r$,
\[
s(p^r)\le r-k
\iff
B_{r-1}=B_{r-2}=\cdots=B_{r-k}=0.
\]
\end{lemma}

\begin{proof}
The same argument as in Lemma~\ref{lem:upper} applies here from $n\ge0$: there the
restriction $n\ge r$ was needed to eliminate from $F_n$ the terms with $p\mid k$,
whereas $U_n$ is already defined using only the terms with $p\nmid k$. Thus
$U_{n+P}=U_n$ for every $n\ge0$, while Proposition~\ref{prop:expansion}\textup{(iv)} gives $F_n=U_n$ for $n\ge r$, and hence $B_n=0$ for $n\ge r$.
Since $P=p^{r-1}\operatorname{lcm}(p-1,4)\ge 4p^{r-1}\ge 4\cdot3^{r-1}>r$, for
$0\le n\le r-1$ we have $B_{n+P}=0$ and
$F_{n+P}=U_{n+P}=U_n$, while $F_n=U_n+B_n$, hence
$F_{n+P}=F_n\iff B_n=0$ whenever $0\le n\le r-1$.  Equivalently, writing
$n=r-j$ for $1\le j\le k$, which is
$n=r-1,r-2,\ldots,r-k$ , and since the period already holds for all $n\ge r$, 
Corollary~\ref{cor:preperiod} gives $s(p^r)\le r-k\iff B_{r-1}=B_{r-2}=\cdots=B_{r-k}=0$.
\end{proof}

Thus, the above lemma naturally suggests that the preperiod may be determined by moving backwards through
$B_{r-1},B_{r-2},\dots,B_0$. In particular, $B_{r-1}\ne0$ implies
$s(p^r)=r$ for every $r\ge1$. For $r\ge2$, if, for some
$1\le k\le r-1$, $B_{r-1}=\cdots=B_{r-k}=0$ but
$B_{r-k-1}\ne0$, then $s(p^r)=r-k$.Finally, if
$B_{r-1}=B_{r-2}=\cdots=B_0=0$, then Lemma~\ref{lem:Bcriterion} gives
$s(p^r)\le0$, and since the preperiod is nonnegative, we have $s(p^r)=0$.

We now examine $B_{r-k}$ more closely for $r\ge1$ and $1\le k\le r$. Since $H_0$ and the coefficients $H_j$ for $j\ge1$ are different in Proposition~\ref{prop:expansion}\textup{(ii)}, and the $j=0$ term contributes when $k=r$, we treat this endpoint separately. Thus, the pairs $(r,k)$ split into two cases: either $k=r$, or $1\le k\le r-1$. In the latter case no such $k$ exists when $r=1$, so this case is equivalently $r\ge2$ and $1\le k\le r-1$. We first consider this latter case. Since in that case $r-k\ge1$, the $j=0$ term in the definition of $B_{r-k}$ vanishes, and hence we have $B_{r-k}=\sum_{\substack{1\le j\le p^r-1\\ p\mid j}}H_j(ji)^{r-k}$. Writing $j=p\ell$ and using $H_{p\ell}=2i^{-p\ell}\eta^{-1}$ from Proposition~\ref{prop:expansion}\textup{(ii)} gives
$B_{r-k}=p^{r-k}u_{r,k}T_{r,k}$, where
$u_{r,k}:=2\eta^{-1}i^{r-k}$ and
$T_{r,k}:=\sum_{\ell=1}^{p^{r-1}-1}\ell^{r-k}i^{-p\ell}$.
By Lemma~\ref{lem:Sr}\textup{(iii),(iv)} and Proposition~\ref{prop:expansion}\textup{(i)},
$u_{r,k}\in S_r^\times$, and hence
$B_{r-k}=0\iff p^{r-k}T_{r,k}=0$.

Recall that in our previous analysis we passed to $S_r/pS_r$, since we wanted both
a general criterion for determining whether $L$ is a period, as in the discussion
preceding Lemma~\ref{lem:bn}, namely $\bar X_n=\bar0$, and a setting in which the
Vandermonde argument could be applied. Thus, passing to the quotient gave us something
intuitively useful to work with, which encouraged us again to try applying the same idea here
as well. By using Lemma~\ref{lem:Sr}\textup{(v)}, we pass to $S_r/p^kS_r$ and examine what
happens to $T_{r,k}$. Rather surprisingly, the resulting expression exhibits a
self-similarity with the frequency expansion used for $S_k$. The remaining case
$k=r$, with $r\ge1$, will be treated separately. We can now
present our findings.

\begin{lemma}\label{lem:Bdiv}
For $1\le k\le r-1$ and $r\geq2$
$$
B_{r-k}=0\iff p^k\mid E_{r-k}.
$$
\end{lemma}

\begin{proof}  
Using Lemma~\ref{lem:Sr}\textup{(v)},
\[
B_{r-k}=0
\iff p^{r-k}T_{r,k}=0
\iff T_{r,k}\in p^kS_r
\iff \overline{T}_{r,k}=\bar0
\quad\text{in }S_r/p^kS_r.
\]
Put $n=r-k$ and  
$q=p^{r-1-k}$. Since $p^{r-1}=qp^k$ and $n\ge1$, adjoining the zero  
term and grouping $\ell=a+jp^k$ yields  
\[  
\overline{T}_{r,k}  
=  
\left(\sum_{j=0}^{q-1}\bar i^{-p^{k+1}j}\right)  
\left(\sum_{a=1}^{p^k-1}\bar a^{\,n}\bar i^{-pa}\right).  
\]  
Abbreviate $w:=\bar i^{-p^{k+1}}$. Since $w\in\{\pm\bar i\}$, we have \(w^4=\bar1\) and \(\bar1+w+w^2+w^3=\bar0\). As \(q\) is odd, \(\sum_{j=0}^{q-1}w^j\) equals either \(\bar1\) or \(w\), according as \(q\equiv1\) or \(3\pmod4\). Hence the factor $\sum_{j=0}^{q-1}w^j$ lies in \(\{\bar1,\pm\bar i\}\), all of which are units since \(1,\pm i\in S_r\) are units by Lemma~\ref{lem:Sr}\textup{(iii)}, and the quotient homomorphism sends units to units.
 If $p\equiv1\pmod4$, then $\bar i^{-pa}=\bar i^{-a}$, so no reindexing is needed;   
if $p\equiv3\pmod4$, then $\bar i^{-pa}=\bar i^a$, and reindexing $a\mapsto p^k-a$ gives
$\sum_{a=1}^{p^k-1}\bar a^n\bar i^a
=\sum_{a=1}^{p^k-1}\overline{(p^k-a)}^{\,n}\bar i^{\,p^k-a}$.
Since $\overline{(p^k-a)}^{\,n}=(-\bar1)^n\bar a^n$ and
$\bar i^{\,p^k-a}=\bar i^{\,p^k}\bar i^{-a}$, the second factor changes only by the unit
$(-\bar1)^n\bar i^{\,p^k}$, again by Lemma~\ref{lem:Sr}\textup{(iii)} and the same observation. Hence, in both cases we have
\[  
\overline{T}_{r,k}=\bar0  
\iff   
\sum_{a=1}^{p^k-1}\bar a^{\,n}\bar i^{-a}=\bar0   
\]  
Surprisingly, the expression on the right is exactly of the same form as the sum appearing in the
$n$-th coefficient, for $n\ge1$, of the frequency expansion in
Proposition~\ref{prop:expansion}\textup{(iii)} with $r=k$, so that the ambient
ring is $S_k$. Indeed, writing $F_n^{(k)}$ and $\eta_k$ for the corresponding
objects in $S_k$, taking the $n$-th coefficient in
Proposition~\ref{prop:expansion}\textup{(iii)}, and noting that the zero-frequency
term $a=0$ vanishes since $n\ge1$, gives
$F_n^{(k)}=2\eta_k^{-1}i^n\sum_{a=1}^{p^k-1}a^ni^{-a}$.
To make this connection precise, the natural reduction $S_r\to S_k$ is surjective with kernel
$p^kS_r$, so the First Isomorphism Theorem \cite[p.~142]{Aluffi}
gives $S_r/p^kS_r\cong S_k$. Under this isomorphism, and since the prefactor $2\eta_k^{-1}i^n$ is a unit by
Lemma~\ref{lem:Sr}\textup{(iii),(iv)} and
Proposition~\ref{prop:expansion}\textup{(i)}, we have
\[
F_n^{(k)}=0
\iff
\sum_{a=1}^{p^k-1}a^ni^{-a}=0\quad\text{in }S_k
\iff
\sum_{a=1}^{p^k-1}\bar a^{\,n}\bar i^{-a}=\bar0
\quad\text{in }S_r/p^kS_r
\iff
\overline{T}_{r,k}=\bar0.
\]
Combining this with the preceding equivalences gives
\[
B_{n}=0\iff F_n^{(k)}=0\iff [E_n]_{p^k}=0\iff p^k\mid E_n.
\]
Since $n=r-k$, the result follows.
\end{proof}
We have now analyzed the case $r\ge2$ and $1\le k\le r-1$. We next consider the remaining case $k=r$, with $r\ge1$. In this case $B_{r-k}=B_0$ for every $r\ge1$, so we first evaluate $B_0$ in the following lemma.

\begin{lemma}\label{lem:B0}
For every odd prime $p$ and $r\ge1$, one has $B_0\in S_r^\times$; in particular,
$B_0\ne0$.
\end{lemma}

\begin{proof}
By Proposition~\ref{prop:expansion}\textup{(ii)} and above definition of $B_n$, one has
$B_0=H_0+\sum_{\ell=1}^{p^{r-1}-1}H_{p\ell}
=\eta^{-1}\left(1+i^{-p^r}+2\sum_{\ell=1}^{p^{r-1}-1}i^{-p\ell}\right)$.
Put $M:=p^{r-1}$ and $z:=i^{-p}\in S_r$, where $\eta=i(z^M-1)$ is a unit by Proposition~\ref{prop:expansion}\textup{(i)}. Since $p$ is odd,
$z\in\{\pm i\}$, so $1+z,1-z\in\{1+i,1-i\}$ are units by
Lemma~\ref{lem:Sr}\textup{(iv)}. Moreover, $z^M=i^{-p^r}$,$z^{\ell}=i^{-p\ell}$ and
$1+z^M+2\sum_{\ell=1}^{M-1}z^\ell
=(1+z)\sum_{\ell=0}^{M-1}z^\ell
=(1+z)(1-z^M)(1-z)^{-1}$.
Hence
$B_0=\eta^{-1}(1+z)(1-z^M)(1-z)^{-1}
=i(1+z)(1-z)^{-1}$, then we conclude $B_0\in S_r^\times$.
\end{proof}

Recall that $E_0=E_1=E_2=1$, $E_3=2$, $E_4=5$, $E_5=16$, and $E_6=61$, as listed in OEIS~\cite{OEIS111}. Lemma~\ref{lem:B0} shows that $B_0\ne0$. We know $p^r\nmid E_0$ and when $k=r$ for $r\geq1$ the case
$B_0=0\iff p^r\mid E_0$ holds. Thus, considering this with Lemma~\ref{lem:Bdiv}, for $r\geq1$  we have
$B_{r-k}=0\iff p^k\mid E_{r-k}$ for every $1\le k\le r$.
When combined with Lemma~\ref{lem:Bcriterion} we have for every $r\ge1$ and $1\le k\le r$,
\[
s(p^r)\le r-k
\iff
p\mid E_{r-1},\quad p^2\mid E_{r-2},\quad\ldots,\quad p^k\mid E_{r-k}.
\]

In the same way as we did above for the terms $B_{r-k}$, we can now determine the
preperiod directly from these divisibility conditions. For every $r\ge1$, if
$p\nmid E_{r-1}$, then $s(p^r)=r$. For $r\ge2$, if, for some
$1\le k\le r-1$, $p\mid E_{r-1},p^2\mid E_{r-2},\ldots,p^k\mid E_{r-k}$ but
$p^{k+1}\nmid E_{r-k-1}$, then $s(p^r)=r-k$. We can therefore summarize
the possibilities in the following table.
\begin{table}[H]
\centering
\caption{Preperiod determination table for a given $r\ge1$}
\label{tab:preperiod}
\[
\left.
\begin{array}{c|c}
\text{condition} & \text{preperiod}\\ \hline
p\nmid E_{r-1}
& s(p^r)=r\\[2mm]

p\mid E_{r-1},\quad p^2\nmid E_{r-2}
& s(p^r)=r-1\\[2mm]

p\mid E_{r-1},\quad p^2\mid E_{r-2},\quad p^3\nmid E_{r-3}
& s(p^r)=r-2\\[2mm]

\vdots
& \vdots\\[2mm]

p\mid E_{r-1},\quad p^2\mid E_{r-2},\quad\ldots,\quad
p^k\mid E_{r-k},\quad p^{k+1}\nmid E_{r-k-1}
& s(p^r)=r-k\\[2mm]

\vdots
& \vdots\\[2mm]

p\mid E_{r-1},\quad p^2\mid E_{r-2},\quad\ldots,\quad
p^{r-1}\mid E_1,\quad p^r\nmid E_0
& s(p^r)=1\\[2mm]

p\mid E_{r-1},\quad p^2\mid E_{r-2},\quad\ldots,\quad
p^{r-1}\mid E_1,\quad p^r\mid E_0
& s(p^r)=0
\end{array}
\hspace{1.5em}
\right\}
\qquad r+1\text{ rows}
\]
\end{table}
Let us explain how the table is to be read. Fix $r\ge1$. There are then $r+1$ possible cases; for instance, when $r=4$, the table has five rows, corresponding to the possible preperiods $4,3,2,1,0$. We begin with the first row and check its nondivisibility condition. If it is satisfied, the preperiod is determined. If not, the corresponding divisibility condition holds, which allows us to move to the next row. Repeating this procedure, we proceed down the table until the first nondivisibility condition is met. Moreover, only the rows for which all indices are defined are relevant. For example, when $r=2$, the possible preperiods are $2,1,0$, so the case $r-2=0$ is given by the final row, not by the displayed schematic $r-2$ row, which would produce the undefined term $E_{-1}$. In general, any row producing an index $j<0$ in $E_j$ is simply not applicable for that value of $r$.The remaining nuances are left to the reader, as they are clear from the general pattern of the table.

The final row is included only for completeness, so that all formally possible values of the preperiod are displayed. In fact, this case never occurs. Since $E_0=1$, we always have $p^r\nmid E_0$. Thus the procedure necessarily stops at or before the row corresponding to $s(p^r)=1$, and consequently $s(p^r)\ge1$.

\begin{theorem}[Ramassamy's preperiod conjecture]\label{thm:smallest-counterexample}
The conjecture $s(p^r)=r$ is false. Its smallest
counterexample among odd prime powers is
\[
5^5=3125,
\]
for which $s(5^5)=4$.
\end{theorem}

\begin{proof}
For $r=1$, we have $p\nmid E_0=1$, so the criterion above gives
$s(p)=1$, and hence no counterexample occurs. 
Since $E_1=E_2=1$ and $E_3=2$, none is divisible by an odd prime, so
no counterexample is possible for $r=2,3,4$. For $r=5$, we have
$E_4=5$, so $p=5$ gives $5\mid E_4$ and moreover, $5^2\nmid E_3=2$, so the table above gives
\[
s(5^5)=4.
\]

It remains to verify minimality. Having excluded $r\le4$, suppose
$p^r<5^5$ with $r\ge5$. Then necessarily $p=3$, since $p\ge5$ would
give $p^r\ge5^5$. As $3^7<5^5<3^8$, the only remaining cases are
$3^5,3^6,$ and $3^7$. However, $3\nmid E_4=5$, $3\nmid E_5=16$, and
$3\nmid E_6=61$, so none is a counterexample. Hence $5^5$ is the
smallest odd prime power for which $s(p^r)\ne r$.
\end{proof}

The divisibility criterion table developed above is particularly well suited
to computational searches, since it reduces the determination of the
preperiod to a sequence of elementary integer divisibility tests.
Ramassamy carried out a computational search over odd prime powers
$p^r\le10^3$ \cite{R}, which is also why he missed the first
counterexample, $3125$. We first extended this search to all odd prime
powers $p^r\le10^{300}$. In this range, we found $461$ exceptions to
the conjectured equality $s(p^r)=r$; in every such case, however, the
preperiod was $r-1$.

This naturally led us to ask whether there are cases in which the
preperiod is $r-2$. A search designed specifically to detect such cases
found no example up to $10^{1500}$. Extending the range to $10^{2000}$,
however, produced the example $s(43^{980})=978$.
We then asked whether there are cases with preperiod at most $r-3$. Rather
interestingly, an exhaustive search found no such example among odd
prime powers $p^r\le10^{100000}$.
Therefore, these computations suggest that the phenomenon may be universal,
leading us to the following conjecture.
\begin{conjecture}\label{conj:preperiod-lower}
For every odd prime $p$ and every $r\ge2$,
\[
s(p^r)\ge r-2.
\]
\end{conjecture}

\begin{conjecture}
If Conjecture~\ref{conj:preperiod-lower} is false, then for every
integer $k\ge3$ there exist an odd prime $p$ and an integer $r\ge 2$ such that
\[
s(p^r)=r-k.
\]
\end{conjecture}

We leave this conjecture as an open problem.
\section*{Declaration on the use of generative AI}

Generative AI models (Anthropic Claude Fable 5 and Claude Opus 5;
OpenAI GPT-5.6 Sol) were used extensively in the preparation of this
manuscript. The mathematics is the author's: the author developed the
results and the proof ideas, determined the individual steps of the
arguments and the route to be followed at each stage, and supplied the key
devices on which the proofs turn, as well as the overall organization of
the paper and the arrangement of the sections and lemmas. This material was
given to the models in the form of instructions and sketches, pen and paper drafts,  together with
detailed directions concerning its presentation.

The models worked this material up into finished written form. The proofs
as they appear here were composed by the models from the author's ideas,
steps and instructions; the remaining prose of the manuscript, the \LaTeX{}
source and the typesetting were generated by the models from draft material supplied by the author. The computational analysis of the search for counterexamples, based on the method and theoretical framework developed by the author and summarized in Table~\ref{tab:preperiod} associated with Section~\ref{sec:preperiod}, was likewise performed by the models under the author's direction and instructions.

The author has independently reviewed and verified the mathematical
statements, proofs, computations, references, and code presented in the
manuscript, and takes full responsibility for their accuracy and for the
entire contents of the paper.

\end{document}